\documentclass[final]{siamonline190516}
\usepackage{graphicx,subfigure}
\usepackage{amsfonts,amsmath,amssymb,mathtools}
\usepackage{bm}
\usepackage{color,xcolor}
\usepackage{appendix}
\usepackage[]{hyperref}  
\usepackage{comment}
\usepackage{csquotes}
\usepackage{enumerate}
\usepackage[shortlabels]{enumitem}
\setlist[enumerate]{leftmargin=.5in}
\setlist[itemize]{leftmargin=.5in}

\newcommand{\id}{\mathrm d}

\newcommand{\vc}{\mathbf}

\newcommand{\pard}[2]{\frac{\partial #1}{\partial #2}}

\DeclarePairedDelimiter\ceil{\lceil}{\rceil}

\newtheorem{thm}{Theorem}
\newtheorem{prop}{Proposition}
\newtheorem{rem}{Remark}
\newtheorem{lem}{Lemma}

\title{Multiscale analysis of accelerated gradient methods\thanks{Accepted for publication in SIAM Journal on Optimization}}
\author{Mohammad Farazmand\thanks{Department of Mathematics,
North Carolina State University,
2311 Stinson Dr., Raleigh, NC 27695-8205, USA. Tel: +1 919-515-2598. Email:\email{farazmand@ncsu.edu}.}
}
\headers{Multiscale analysis of accelerated methods}{M. Farazmand}

\begin{document}
\maketitle

\begin{abstract}
Accelerated gradient descent iterations are widely used in optimization.
It is known that, in the continuous-time limit, these iterations 
converge to a second-order differential equation which we refer to as the accelerated gradient flow.
Using geometric singular perturbation theory, we show 
that, under certain conditions, the accelerated gradient flow possesses an attracting
invariant slow manifold to which the trajectories of the flow 
converge asymptotically. We obtain a general explicit expression in the form of functional series expansions
that approximates the slow manifold to any arbitrary order of accuracy.
To the leading order, the accelerated gradient flow
reduced to this slow manifold coincides with the usual gradient descent.
We illustrate the implications of our results on three examples.
\end{abstract}

\section{Introduction}
We consider the convex optimization problem $\min_{\vc x\in\mathcal X} f(\vc x)$ where $f:\mathcal X\to \mathbb R$
is a convex function of class $C^r$ (i.e. $f$ is $r$-time continuously differentiable) and
$\mathcal X\subseteq \mathbb R^n$ is a compact convex set. We also assume 
that $f$ attains its unique minimum at a point $\vc x^\ast \in\mathcal X$.

The most well-known method for obtaining the minimum $\vc x^\ast$ is 
the gradient (or steepest) descent iterations
\begin{equation}
\vc x_{k+1}=\vc x_k - s \nabla f(\vc x_k),
\label{eq:gd}
\end{equation}
where $s>0$ is a parameter. Under certain assumptions, the sequence $\{\vc x_k\}$ is guaranteed to
converge to the minimum $\vc x^\ast$~\cite{boyd2004}. It is also well-known that in the limit $s\to 0$, 
the discrete iterations~\eqref{eq:gd} converge to the continuous-time gradient flow $\dot{\vc x} = -\nabla f(\vc x)$
which is a first-order ordinary differential equation (ODE). In fact, the gradient descent iterations
are an explicit Euler discretization of the gradient flow~\cite{smale1981,faraz_adjoint}.

The convergence rate of the gradient descent iterations is $\mathcal O(1/k)$ which is rather slow~\cite{nesterov2013}. 
In order to improve this convergence rate, a number of accelerated gradient descent iterations have been 
developed~\cite{polyak1964}.
Most notable perhaps is Nesterov's accelerated gradient descent,
\begin{equation}
\vc x_{k+1}=\vc x_k+\lambda_k(\vc x_k-\vc x_{k-1}) - s_k \nabla f(\vc x_k+\lambda_k(\vc x_k-\vc x_{k-1})),
\label{eq:nestrov}
\end{equation}
where $\{\lambda_k\}$ and $\{s_k\}$ are prescribed sequences of positive real numbers~\cite{nesterov1983}.
The convergence rate of this accelerated gradient descent is $\mathcal O(1/k^2)$.
The terms involving $\lambda_k$ are referred to as the acceleration terms. 
In practice, $\lambda_k$ and $s_k$ are sometimes assumed to be constants independent of $k$.
For $\lambda_k=0$, we recover the steepest descent iterations~\eqref{eq:gd}.

Su et al.~\cite{su2014} discovered that, in a small step size limit, the Nesterov iterations
converge to the second-order differential equation,
\begin{equation}
\ddot{\vc x}+\frac{\rho}{t}\dot{\vc x}+\nabla f(\vc x)=0,
\label{eq:nesterov_cont}
\end{equation}
where $\rho =3$. We refer to this differential equation as the \emph{Nesterov flow}.
Recall that the gradient descent iterations converge, in the continuous-time limit, to the
autonomous first-order differential equation $\dot{\vc x} = -\nabla f(\vc x)$
while the Nesterov's accelerated gradient descent converges to the
non-autonomous second-order differential equation~\eqref{eq:nesterov_cont}
(It is non-autonomous because of the explicit dependent of 
the coefficient $\rho/t$ on time $t$).

More recently, Wibisono et al.~\cite{wibisono2016} showed that a large class of accelerated gradient methods 
can be formulated as temporal discretizations of a second-order differential equation.
In their framework, the continuous-time accelerated gradient methods coincide with the Euler--Lagrange equations corresponding 
to a single Lagrangian, 
\begin{equation}
\mathcal L(\vc x,\dot{\vc x},t)=e^{\alpha_t+\gamma_t}\left( D_h(\vc x+e^{-\alpha_t}\dot{\vc x},\vc x)-e^{\beta_t}f(\vc x)\right),
\label{eq:BregmanLag}
\end{equation}
which we refer to as the \emph{Bregman Lagrangian}.
Here, $\alpha_t$, $\beta_t$ and $\gamma_t$ are potentially time-dependent scalar functions and 
the Bregman divergence $D_h$ is defined as
\begin{equation}\label{eq:bregman_div}
D_h(\vc y,\vc x)=h(\vc y)-h(\vc x)-\langle \nabla h(\vc x),\vc y-\vc x\rangle,
\end{equation}
where $\langle\cdot,\cdot\rangle$ denotes the usual Euclidean inner product and
the \emph{distance-generating function} $h:\mathcal X\to \mathbb R$ is convex and continuously differentiable.

Wibisono et al.~\cite{wibisono2016} assume the so-called ideal scaling, $\dot{\gamma}_t=e^{\alpha_t}$, 
and show that the Euler--Lagrange equation corresponding to
the Bregman Lagrangian~\eqref{eq:BregmanLag} reads
\begin{equation}
\ddot{\vc  x} +\left(e^{\alpha_t}-\dot{\alpha}_t\right) \dot{\vc  x} + e^{2\alpha_t+\beta_t}
\left[ \nabla^2 h(\vc x+e^{-\alpha_t}\dot{\vc x})\right]^{-1} \nabla f(\vc x) = 0. 
\label{eq:EL}
\end{equation}
For the special choice $\alpha_t = \log((\rho-1)/t)$, $\beta_t = -2\log((\rho-1)/t)$ and
$h(\vc x)=\frac12 \|\vc x\|^2$, the Euler--Lagrange equation~\eqref{eq:EL} coincides with 
equation~\eqref{eq:nesterov_cont}.

Wibisono et al.~\cite{wibisono2016} also show that, for certain functions $\alpha_t$ and $\beta_t$,
the trajectories of~\eqref{eq:EL} converge to the minimum $\vc x^\ast$ asymptotically. Furthermore, they recover 
several well-known accelerated gradient methods by discretizing equation~\eqref{eq:EL}
in time in an appropriate fashion.

In this paper, we investigate the multiscale behavior of the accelerated gradient flows
and their phase space structure using geometric singular perturbation theory.
We divide our analysis into two parts: autonomous and non-autonomous cases. 
In the autonomous case (Section~\ref{sec:slowManRed_aut}), the choice of the functions $\alpha_t$ and $\beta_t$ 
is restricted such that
the Euler--Lagrange equation~\eqref{eq:EL} does 
not have an explicit dependence on time. 
In this case, we show that the solutions of 
the Euler--Lagrange equation~\eqref{eq:EL} converge exponentially fast towards an 
$n$-dimensional, attracting, slow manifold embedded in the $2n$-dimensional phase space 
of the Euler--Lagrange equation. We derive explicit formulas for the slow manifold
in terms of a functional series expansion. 
To the leading order, the Euler--Lagrange equation, reduced to the slow manifold, 
coincides with the usual gradient descent. 
We also investigate the reduced flow at higher orders and prove 
that the minimum $\vc x^\ast$ is a locally asymptotic stable fixed point of the reduced flow
at any order.

The Nesterov accelerated gradient flow~\eqref{eq:nesterov_cont}
is non-autonomous because of the explicit time dependence of the 
coefficient $\rho/t$. We treat this non-autonomous case separately (Section~\ref{sec:slowManRed_nonaut}) using a non-autonomous extension
of the singular perturbation theory. While the results are similar to the autonomous 
accelerated gradient flow, here the slow manifold is an $(n+1)$-dimensional
invariant manifold in the $(2n+1)$-dimensional \emph{extended} phase space of the flow.

The paper is organized as follows.
Section~\ref{sec:slowManRed_aut} contains our main results including the
slow manifold reduction in Euclidean and non-Euclidean spaces
in the autonomous case. Section~\ref{sec:slowManRed_nonaut}
treats the non-autonomous case of equation~\eqref{eq:nesterov_cont}.
In section~\ref{sec:examples}, 
we demonstrate the implications of our results on three examples. Section~\ref{sec:conc} contains
our concluding remarks.

\section{Multiscale analysis -- Autonomous case}\label{sec:slowManRed_aut}
\subsection{Euclidean case}
In this section, we consider the Euclidean case where the distance-generating
function in the Bregman Lagrangian~\eqref{eq:BregmanLag} is given by $h(\vc x)=\frac12 \|\vc x\|^2$. In this case, 
the Lagrangian reads
\begin{equation}
\mathcal L(\vc x,\dot{\vc  x},t)=e^{\alpha_t+\gamma_t}\left( \frac12 \|e^{-\alpha_t}\dot{\vc x}\|^2-e^{\beta_t}f(\vc x)\right).
\end{equation}
The corresponding Euler--Lagrange equation reads
\begin{equation}
\ddot{\vc x} +\left(e^{\alpha_t}-\dot{\alpha}_t\right) \dot{\vc x} + e^{2\alpha_t+\beta_t}\nabla f(\vc x) = 0. 
\label{eq:EL_eucl}
\end{equation}

The following theorem, which constitutes our main result, states that under certain assumptions the trajectories of the Euler--Lagrange 
equation~\eqref{eq:EL_eucl} converge asymptotically to an $n$-dimensional invariant slow manifold embedded in the $2n$-dimensional phase space
$(\vc x,\dot{\vc x})$. This slow manifold is a graph over the $\vc x$ coordinates. We obtain an explicit expression for this graph 
in the form of a functional series expansion. 
Our results use Fenichel's geometric singular perturbation theory~\cite{fenichel1979geometric}.
There are various statements of this theory with different extents of generality~\cite{wiggins1994normally,jones1995,Kuehn2015}. 
A statement of the geometric singular perturbation theory that most closely 
fits our purposes can be found in Chapter 3 (and Theorem 3.1.4.) of~\cite{Kuehn2015}.

\begin{thm}\label{thm:slowFlow_eucl}
Let $f\in C^{r+1}(\mathcal X)$ with $r\geq 0$.
Define $\mu =e^{\alpha_t}-\dot{\alpha}_t$ and $\eta = e^{2\alpha_t+\beta_t}$. If $\mu$ and 
$\eta$ are constant, then there exists $\mu_0>0$ such that for all $\mu >\mu_0$ the following are true.
\begin{enumerate}[(i)]
\item The trajectories of the Euler--Lagrange equation~\eqref{eq:EL_eucl}
converge exponentially fast to an $n$-dimensional invariant manifold embedded 
in the $2n$-dimensional phase space $(\vc x,\dot{\vc  x})$.
Furthermore, this invariant manifold is a graph over the $\vc x$ coordinates 
(See figure~\ref{fig:schem_gsp} for an illustration).

\item The flow of~\eqref{eq:EL_eucl} on the invariant manifold $\mathcal M_\epsilon$ is given by
\begin{equation}\label{eq:reducedFlow}
\dot{\vc  x} = \sum_{k=0}^{r}\epsilon^{2k+1}\vc g_{2k+1}(\vc x)+\mathcal O(\epsilon^{2r+3}),
\end{equation}
where $\epsilon=\mu^{-1}$ and the maps $\vc g_k:\mathcal X\to\mathbb R^n$ are defined recursively by
\begin{align}
\vc g_1(\vc x) &= -\eta \nabla f(\vc x),\nonumber\\
\vc g_{2k}(\vc x) &=\vc  0, \nonumber\\
\vc g_{2k+1}(\vc x) &= -\sum_{\ell=1}^{2k-1}\nabla\vc  g_\ell(\vc x)\vc g_{2k-\ell}(\vc x),\quad k\geq 1.
\label{eq:gk}
\end{align}
\end{enumerate}
\end{thm}

\begin{figure}
\centering
\includegraphics[width=.95\textwidth]{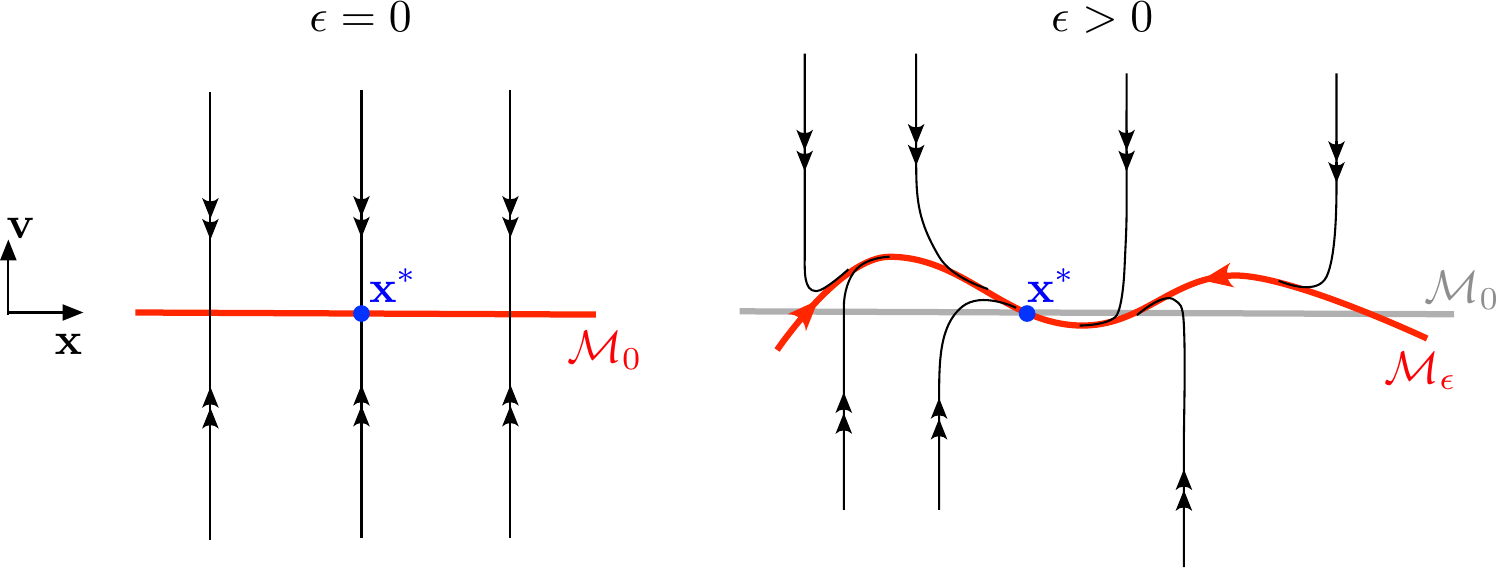}
\caption{A sketch of the geometric singular perturbation theory.
In the singular limit, $\epsilon=0$, the 
set $\mathcal M_0=\{(\vc x,\vc v)\in\mathbb R^{2n}:\ \vc v=0\}$ is filled with fixed points. The trajectories
$(\vc x(\tau),\vc v(\tau))=(\vc x_0,\vc v_0e^{-\tau})$ converge exponentially fast to $(\vc x_0,\vc 0)\in\mathcal M_0$ 
from any initial condition $(\vc x_0,\vc v_0)$. In other words, the set $\mathcal M_0$ is 
invariant and the global attractor of the system. For $0<\epsilon\ll 1$, 
the manifold $\mathcal M_0$ deforms into a nearby manifold $\mathcal M_\epsilon$
which is also invariant and globally attracting. However, $\mathcal M_\epsilon$ 
is not necessarily a collection of fixed points. 
}
\label{fig:schem_gsp}
\end{figure}

Before presenting the proof of this theorem, we make a few remarks.
\begin{rem}
The conditions requiring $\mu  = e^{\alpha_t}-\dot{\alpha}_t$ and $\eta = e^{2\alpha_t+\beta_t}$
to be constant are equivalent to
\begin{equation}
\alpha_t = \log \left[ \frac{\mu}{1+\left(\mu e^{-\alpha_0}-1\right) e^{\mu t}}\right],
\quad\beta_t = \log\eta -2\alpha_t, 
\label{eq:alphabeta}
\end{equation}
where $\alpha_0,\eta\in\mathbb R$ and $\mu>0$ are arbitrary constants satisfying $\alpha_0\leq \log \mu$.
For instance, the class of accelerated gradient flows considered in Eq. (7) of Ref.~\cite{wilson2016} 
and in Eq. (4) of Ref.~\cite{jin18a} satisfy these conditions.
We point out that these conditions are sufficient for the results of Theorem~\ref{thm:slowFlow_eucl} 
to hold but are not necessary and can possibly be relaxed.
\label{rem:alphabeta}
\end{rem}

\begin{rem}
Note that the minimum $\vc x^\ast$ is a fixed point of~\eqref{eq:reducedFlow} since
$\nabla f(\vc x^\ast)=\vc g_k(\vc x^\ast)=\vc 0$ for all $k\in\mathbb N$.
\end{rem}

\begin{rem}
Polyak's heavy ball iterations are a particular discretization of the ODE~\eqref{eq:EL}
with constant $\mu =e^{\alpha_t}-\dot{\alpha}_t$ and $\eta = e^{2\alpha_t+\beta_t}$~\cite{polyak1964}. 
The convergence of the Polyak's heavy ball iterations is only guaranteed if $f$ is strongly convex~\cite{lessard2016}. 
However, the fixed point $(\vc x,\dot{\vc x})=(\vc x^\ast,\vc 0)$ is a globally asymptotically stable fixed point of the ODE~\eqref{eq:EL_eucl}
for all $\mu>0$ and $C^1$ convex functions $f$.
One can verify this by considering the Lyapunov function
\begin{equation}
\mathcal E(\vc x,\dot{\vc x}) = \eta \left[ f(\vc x)-f(\vc x^\ast)\right] + \frac12 |\dot{\vc x}|^2.
\end{equation}
Note that $\mathcal E(\vc x,\dot{\vc x})>0$ for all $(\vc x,\dot{\vc x})\in\mathcal X\times \mathbb R^n\backslash (\vc x^\ast,\vc 0)$ and 
$\frac{\id}{\id t}\mathcal E = - \mu |\dot{\vc x}|^2\leq 0$. By LaSalle's invariant set theorem~\cite{LaSalle1960}, and the fact that $\mathcal E(\vc x,\dot{\vc x})\to \infty$
as $|(\vc x,\dot{\vc x})|\to\infty$, the fixed point $(\vc x^\ast, \vc 0)$ is globally asymptotically stable.
The fact that Polyak's iterations fail to converge for certain convex functions
is an artifact of the discretization.
\end{rem}

\begin{rem}\label{rem:gradient}
To the leading order (neglecting $\mathcal O(\epsilon^3)$ terms), equation~\eqref{eq:reducedFlow}
reduces to the usual gradient flow
\begin{equation}
\dot{\vc  x} = -\epsilon\eta \nabla f(\vc x),
\end{equation}
with a rescaling of time with the constant $\epsilon\eta$. This connection to the gradient flow
was already speculated in an appendix of Ref.~\cite{wibisono2016} through a  heuristic 
argument.
\end{rem}

\begin{proof}[Proof of Theorem~\ref{thm:slowFlow_eucl}]
We first rewrite equation~\eqref{eq:EL_eucl} as a system of first-order differential equations by introducing a new variable $\vc v =\dot{\vc  x}$,
so that
\begin{align}
\dot{\vc x} = & \vc v,\nonumber\\
\dot{\vc v}= & -\mu \vc v -\eta \nabla f(\vc x). 
\label{eq:EL_1st}
\end{align}
Next, we introduce the fast time $\tau =t/\epsilon$ where $\epsilon=\mu^{-1}\ll 1$.
In terms of the fast time $\tau$, equations~\eqref{eq:EL_1st} can be written as 
\begin{align}
\vc x' = & \epsilon \vc v,\nonumber\\
\vc v' = & - \vc v -\epsilon\eta \nabla f(\vc x),
\label{eq:EL_fast}
\end{align}
where prime denotes the derivative with respect to the fast time $\tau$, e.g., $\vc x' =\id \vc x/\id \tau$.

In the singular limit, $\epsilon=0$, equations~\eqref{eq:EL_fast} have the trivial solution
$\vc x(\tau)=\vc x_0$ and $\vc v(\tau)=\vc v_0e^{-\tau}$ where $(\vc x_0,\vc v_0)$ is the initial condition.
The trajectories of the system
converge exponentially fast towards the \emph{critical manifold}
\begin{equation}
\mathcal M_0 = \left\{(\vc x,\vc v)\in\mathcal X\times \mathbb R^{n} :\ \vc v=\vc 0\right\}. 
\end{equation}
By Fenichel's geometric singular perturbation theory~\cite{fenichel1979geometric}, 
there exists $\epsilon_0>0$ such that
for $0<\epsilon<\epsilon_0$,  (i) The trajectories of~\eqref{eq:EL_fast} also converge 
to an $n$-dimensional invariant manifold $\mathcal M_\epsilon$, (ii) Furthermore, $\mathcal M_\epsilon$
is a $C^r$-smooth graph over $\mathcal M_0$ and $\mathcal O(\epsilon)$ close to it. 
More specifically, $\mathcal M_\epsilon$ can be expressed through
the formal series expansion
\begin{equation}
\mathcal M_\epsilon = \left\{(\vc x,\vc v)\in\mathcal X\times \mathbb R^{n} :\ \vc v=\sum_{k=1}^\infty \epsilon^k \vc g_k(\vc x)
\right\},
\label{eq:Meps_formal_series}
\end{equation}
where $\vc g_k:\mathbb R^n\to\mathbb R^n$ are smooth functions to be determined.
As a result, we have 
\begin{align}
\vc v' &= \sum_{k=1}^\infty\epsilon^k \nabla \vc g_k(\vc x)\vc x' \nonumber\\
&= \sum_{k=1}^\infty\epsilon^{k+1} \nabla \vc g_k(\vc x) \vc v\nonumber\\
&= \sum_{k=1}^\infty\sum_{j=1}^\infty\epsilon^{k+j+1} \nabla \vc g_k(\vc x)\vc g_j(\vc x)
\label{eq:vp_1}
\end{align}
On the other hand, equation~\eqref{eq:EL_fast} implies
\begin{align}
\vc v'  &= - \vc v -\epsilon\eta \nabla f(\vc x)\nonumber\\
    &=  -\epsilon \left[ \vc g_1(\vc x) +  \eta \nabla f(\vc x)\right] -\sum_{k=2}^\infty\epsilon^k \vc g_k(\vc x).
\label{eq:vp_2}
\end{align}

Matching the terms of the same order $\epsilon^k$ in equation~\eqref{eq:vp_1}
and~\eqref{eq:vp_2}, we obtain
\begin{align}\label{eq:seq_gk}
\epsilon^1 &: \vc g_1(\vc x) =- \eta \nabla f(\vc x) \nonumber\\
\epsilon^2 &: \vc g_2(\vc x)=\vc 0 \nonumber\\
\epsilon^3 &: \vc g_3(\vc x)= -\nabla \vc g_1(\vc x)\vc g_1(\vc x) \nonumber\\
\epsilon^4 &: \vc g_4(\vc x)=\vc 0 \nonumber\\
\epsilon^5 &: \vc g_5(\vc x)=-\nabla \vc g_3(\vc x)\vc g_1(\vc x)-\nabla \vc g_1(\vc x)\vc g_3(\vc x)\nonumber\\
\epsilon^6 &: \vc g_6(\vc x)=\vc 0 \nonumber\\
\epsilon^7 &: \vc g_7(\vc x)=-\nabla \vc g_5(\vc x)\vc g_1(\vc x)-\nabla \vc g_3(\vc x)\vc g_3(\vc x)-\nabla \vc g_1(\vc x)\vc g_5(\vc x)\nonumber\\
        &\quad \vdots \nonumber\\
\epsilon^k        &: \vc g_k(\vc x) = -\sum_{\ell=1}^{k-2}\nabla \vc g_\ell(\vc x)\vc g_{k-\ell-1}(\vc x),\quad \bmod(k,2)=1.
\end{align}
Note that all functions $\vc g_{k}$ with an even index vanish, i.e., $\vc g_{2k}=\vc 0$ for all $k\in\mathbb N$.
Equation~\eqref{eq:seq_gk} for the odd indices can be rearranged as
\begin{equation}
\vc g_{2k+1}(\vc x) = -\sum_{\ell=1}^{2k-1}\nabla \vc g_\ell(\vc x)\vc g_{2k-\ell}(\vc x),\quad k\geq 1,
\end{equation}
where we have made the change of variables $k\mapsto 2k+1$.
Therefore, the flow on the slow manifold $\mathcal M_\epsilon$ is given by
\begin{equation}
\vc x' = \epsilon \vc v = \epsilon \left[ -\epsilon\eta \nabla f(\vc x)+\sum_{k=1}^{r}\epsilon^{2k+1}\vc g_{2k+1}(\vc x)
+\mathcal O(\epsilon^{2r+3})
\right].
\end{equation}
Rescaling the fast time $\tau$ back to the original time $t$, we obtain equation~\eqref{eq:reducedFlow}.
Note that the function $\vc g_{2r+1}$ involves derivatives of $f$ up to and including order $r+1$.
Since we assumed $f\in C^{r+1}(\mathcal X)$, the formal infinite series~\eqref{eq:Meps_formal_series} must be 
truncated at this order.
\end{proof}

\begin{rem}
Theorem~\ref{thm:slowFlow_eucl} guarantees that the slow manifold persists for $\mu>\mu_0$. It would be
very attractive to estimate the value of $\mu_0$; however, there are two technical difficulties in the way 
of obtaining such estimates. First, $\mu_0$ depends on the function $f$ being minimized. Furthermore, even for a 
given function $f$, obtaining the parameter $\mu_0$ is extremely difficult. The difficulty lies in the fact that, in 
singular perturbation theory, the invariant manifold arises as the fixed point of a certain integral equation~\cite{jones1995,wiggins1994normally}. 
To prove the existence of the fixed point, one uses an infinite-dimensional contraction mapping argument. The contraction
property of this map only holds for sufficiently large $\mu$ (or, equivalently, sufficiently small $\epsilon$). Finding $\mu_0$, 
such that for $\mu>\mu_0$ the contraction property holds, is a tedious task and has only been carried out for extremely simple 
cases (see, e.g., Example 1.3.2. in~\cite{wiggins1994normally}). 
Nonetheless, in the present context, i.e., gradient flow of a convex function, we expect the threshold $\mu_0$ to be small.
The examples shown in Section~\ref{sec:examples} confirm that our invariant slow manifolds persist 
for relatively small values of $\mu$ (see, in particular, the discussion on example 2).
\label{rem:mu0}
\end{rem}

As mentioned in Remark~\ref{rem:gradient}, to the leading order, the reduced flow on the slow manifold is
a gradient flow. It is well-known that the minimizer $\vc x^\ast$ is a locally asymptotically stable fixed point of the 
gradient flow for convex, continuously differentiable functions $f$. The natural question is whether,
for higher order truncations, the minimizer $\vc x^\ast$ remains an asymptotically stable fixed point. 
It is straightforward to show this for the third-order truncations of the reduced flow:

\begin{prop}
Let $f\in C^2(\mathcal X)$ be a convex function.
Up to the third order (neglecting $\mathcal O(\epsilon^5)$ terms), 
the reduced equation~\eqref{eq:reducedFlow} reads
\begin{equation}
\dot{\vc  x} = -\epsilon\eta \nabla f(\vc x) - \epsilon^3\eta^2 \nabla^2f(\vc x)\nabla f(\vc x).
\label{eq:reduced_2ndOrder}
\end{equation}
The minimizer $\vc x^\ast$ is a locally asymptotically stable fixed point of the above equation. 
\end{prop}

\begin{proof}
We use the Lyapunov function
\begin{equation}
\mathcal E(\vc x)=\epsilon\eta\left( f(\vc x)-f(\vc x^\ast)\right)+
\frac12\epsilon^3\eta^2\| \nabla f(\vc x)\|^2.
\end{equation}
Note that $\mathcal E(\vc x)>0$ for all $x\in\mathcal X\backslash \{\vc x^\ast\}$ and that $\mathcal E(\vc x^\ast)=0$.
The gradient of this Lyapunov function is given by
\begin{equation}
\nabla \mathcal E(\vc x) = \epsilon\eta \nabla f(\vc x) + \epsilon^3\eta^2\nabla^2 f(\vc x)\nabla f(\vc x)=-\dot{\vc x},
\end{equation}
and therefore $\frac{\id}{\id t}\mathcal E(\vc x)=\langle \nabla \mathcal E(\vc x),\dot{\vc x}\rangle<0$.
This completes the proof.
\end{proof}

\begin{rem}
Examining equation~\eqref{eq:reduced_2ndOrder}, the reduced equation up to the third order is equivalent to 
a preconditioned gradient descent flow $\dot{\vc x} = -[P(\vc x)]^{-1}\nabla f(\vc x)$ with the 
preconditioner $[P(\vc x)]^{-1} = \epsilon\eta\left[ I + \epsilon^2\eta \nabla^2f(\vc x)\right]$. Since 
$I+\epsilon^2\eta \nabla^2f(\vc x)$ is a near-identity map, its inverse exists for sufficiently small $\epsilon$.
\end{rem}

It turns out that the minimum $\vc x^\ast$ is an asymptotically stable fixed point of the reduced flow 
at any finite order of truncation. This is stated in the following theorem. 
The proof, however, is much more involved and is presented in the appendix. 

\begin{thm}\label{thm:assymptStable}
Let $f\in C^{r+1}(\mathcal X)$ be a convex function.
The point $\vc x=\vc x^\ast$ is an asymptotically stable fixed point of the reduced flow~\eqref{eq:reducedFlow}. 
More precisely, $\vc x=\vc x^\ast$ is an asymptotically
stable fixed point of the differential equation
\begin{equation}\label{eq:reducedFlow_r}
\dot{\vc  x} = \sum_{k=0}^{p}\epsilon^{2k+1}\vc g_{2k+1}(\vc x),
\end{equation}
for all $0\leq p\leq r$ where $\vc g_k:\mathcal X\to\mathbb R^n$ are defined in~\eqref{eq:gk}.
\end{thm}
\begin{proof}
See Appendix~\ref{app:proof_asymStable}.
\end{proof}

\subsection{Non-Euclidean case}
In this section, we consider an important class of distance-generating functions $h$
that appear in the Bregman divergence~\eqref{eq:bregman_div} and have the form
\begin{equation}\label{eq:h_quad}
h(\vc x) = \langle \vc x,\vc x\rangle_H\triangleq \frac12 \langle \vc x,H\vc x\rangle,
\end{equation}
where $\langle\cdot,\cdot\rangle$ denotes the Euclidean inner product and $H$ is a $n\times n$ 
symmetric, positive-definite matrix. Note than $\langle\cdot,\cdot\rangle_H$ defines a
Riemannian metric on the space $\mathbb R^n$.

With the choice~\eqref{eq:h_quad}, the Bregman divergence is given by
\begin{equation}
D_h(\vc y,\vc x) = \frac12 \langle \vc y-\vc x,H(\vc y-\vc x)\rangle,
\end{equation}
and the associated Euler--Lagrange equation~\eqref{eq:EL} reduces to
\begin{equation}
\ddot{\vc  x} +\left(e^{\alpha_t}-\dot{\alpha}_t\right) \dot{\vc x} + e^{2\alpha_t+\beta_t}
H^{-1} \nabla f(\vc x) = 0. 
\label{eq:EL_H}
\end{equation}
Note that, since $H$ is symmetric and positive-definite, its inverse exists.
We have the following result for the slow manifold reduction of~\eqref{eq:EL_H}
which is quite similar to Theorem~\ref{thm:slowFlow_eucl}.

\begin{thm}\label{thm:slowFlow_noneucl}
Let $f\in C^{r+1}(\mathcal X)$ with $r\geq 0$.
Define $\mu =e^{\alpha_t}-\dot{\alpha}_t$ and $\eta = e^{2\alpha_t+\beta_t}$. If $\mu$ and 
$\eta$ are constant, then there exists $\mu_0>0$ such that for all $\mu >\mu_0$ the following are true.
\begin{enumerate}[(i)]
\item The trajectories of the Euler--Lagrange equation~\eqref{eq:EL_H}
converge exponentially fast to an $n$-dimensional invariant manifold embedded 
in the $2n$-dimensional phase space $(\vc x,\dot{\vc  x})$.
Furthermore, this invariant manifold is a graph over the $\vc x$ coordinates 

\item The flow of~\eqref{eq:EL_H} on this invariant manifold is given by
\begin{equation}\label{eq:reducedFlow_H}
\dot{\vc  x} = \sum_{k=0}^{r}\epsilon^{2k+1}\vc g_{2k+1}(\vc x)+\mathcal O(\epsilon^{2r+3}),
\end{equation}
where $\epsilon=\mu^{-1}$ and the maps $\vc g_k:\mathcal X\to\mathbb R^n$ are defined recursively by
\begin{align}
\vc g_1(\vc x) &= -\eta H^{-1}\nabla f(\vc x),\nonumber\\
\vc g_{2k}(\vc x) &= \vc 0, \nonumber\\
\vc g_{2k+1}(\vc x) &= -\sum_{\ell=1}^{2k-1}\nabla\vc  g_\ell(\vc x)\vc g_{2k-\ell}(\vc x),\quad k\geq 1.
\label{eq:gk_H}
\end{align}
\end{enumerate}
\end{thm}

\begin{proof}[Proof of Theorem~\ref{thm:slowFlow_noneucl}]
The proof is quite similar to the Euclidean case (Theorem~\ref{thm:slowFlow_eucl}) and therefore is omitted here
for brevity. 
\end{proof}

We point out that the only difference in the above non-Euclidean result compared to the Euclidean case 
is the appearance of $H^{-1}$ in the definition of $g_1$ in~\eqref{eq:gk_H} which trickles to 
the higher order terms. For instance, we have 
\begin{equation}
\vc g_3(\vc x) = -\eta^2 \left[ H^{-1} \nabla^2 f(\vc x) H^{-1}\right] \nabla f(\vc x).
\end{equation}

Furthermore, Theorem~\ref{thm:assymptStable} also holds in the non-euclidean case. Namely,
the minimizer $\vc x^\ast$ is an asymptotically stable fixed point of~\eqref{eq:reducedFlow_H}
truncated to any order $0\leq p\leq r$. The proof is similar to the Euclidean case and
therefore is omitted here.

\section{Multiscale analysis -- Non-autonomous case}\label{sec:slowManRed_nonaut}

The results in Section~\ref{sec:slowManRed_aut} hold under the key assumptions
that $\mu =e^{\alpha_t}-\dot{\alpha}_t$ and $\eta = e^{2\alpha_t+\beta_t}$ are independent of time. 
This implies that the accelerated gradient flow~\eqref{eq:EL} is an autonomous ODE.
While these assumptions hold for an important subclass of the accelerated gradient flows,
they do not hold for the Nesterov flow~\eqref{eq:nesterov_cont}. In this section, we 
treat the non-autonomous case of the Nesterov flow separately
by applying a time-dependent extension of the singular perturbation theory~\cite{IP_haller08,haragus2010,carvalho2012,roberts2019}.
The results are similar in nature to those of Section~\ref{sec:slowManRed_aut}; however,
in the non-autonomous case, the slow manifold is a graph over $(\vc x,t)$
in the extended phase space $(\vc x,t,\vc v)$.

First, we write the Nesterov flow in the form of an
autonomous ODE in terms of the extended phase space variables $(\vc x,\vc v,t)$
where $\vc v= \dot{\vc x}:=\id \vc x/\id \theta$. Here, $\theta=t-t_0$ is a re-parametrization of 
time.
In the extended phase space, equation~\eqref{eq:nesterov_cont} can be written as a system of first-order differential equations
\begin{equation}
\dot{\vc x}= \vc v,\quad \dot{\vc v} = -\frac{\rho}{t}\vc v -\nabla f(\vc x),\quad \dot t =1,
\label{eq:EL_1st_nonAut}
\end{equation}
with some initial conditions $\vc x(0)=\vc x_0\in\mathcal X$, $\vc v(0)=\vc v_0\in\mathbb R^n$
and $t(0)=t_0>0$. Note that $(\vc x,\vc v,t)$ are functions of the time-like variable $\theta$
and $(\dot{\vc x},\dot{\vc v},\dot t)$ denotes their derivatives with respect to $\theta$.

Next, we introduce the rescaled time $\tau=\theta/\epsilon$ where $\epsilon=\rho^{-1}$.
Denoting the derivatives with respect to the 
fast time $\tau$ with a prime, we can recast equation~\eqref{eq:EL_1st_nonAut} as
\begin{equation}
\vc x' = \epsilon \vc v, \quad \vc v' = -\frac{1}{t}\vc v -\epsilon\nabla f(\vc x),\quad t' =\epsilon. 
\end{equation}
In the singular limit, $\epsilon=0$, we have
\begin{equation}
\vc x(\tau)=\vc x_0, \quad \vc v(\tau)=e^{-\tau/t_0}\vc v_0,\quad t(\tau)=t_0,
\end{equation}
for all $\tau\geq 0$. This implies that the plane 
\begin{equation}
\mathcal M_0 = \{(\vc x,\vc v,t)\in \mathcal X\times\mathbb R^n\times [t_0,t_0+T]: \vc v=\vc 0\}
\end{equation}
is the critical manifold towards which all trajectories converge with the exponential rate $e^{-\tau/t_0}$ (see figure~\ref{fig:SlowMan_NonAut}).

\begin{figure}
\centering
\includegraphics[width=.99\textwidth]{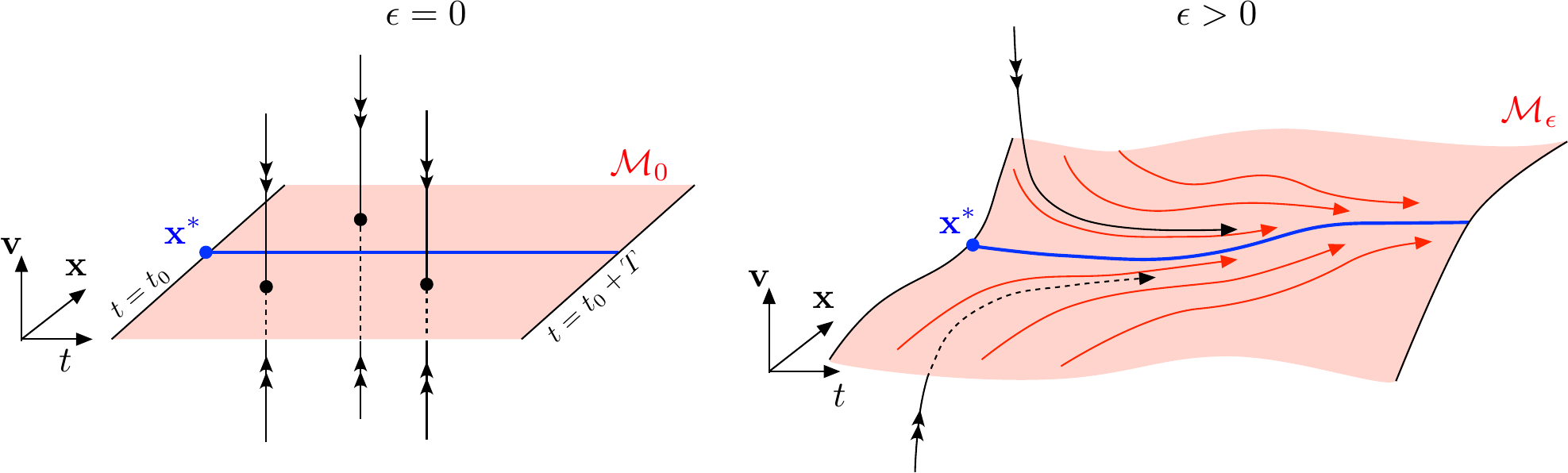}
\caption{
A sketch of the geometric singular perturbation theory for the non-autonomous accelerated gradient flow.
In the singular limit, $\epsilon=0$, the 
set $\mathcal M_0=\{(\vc x,t,\vc v)\in\mathcal X\times [t_0,t_0+T]\times\mathbb R^{n}:\ \vc v=0\}$ is filled with fixed points. The trajectories
$(\vc x(\tau),t(\tau),\vc v(\tau))=(\vc x_0,t_0,\vc v_0e^{-\tau/t_0})$ converge exponentially fast to $(\vc x_0,t_0,\vc 0)\in\mathcal M_0$ 
from any initial condition $(\vc x_0,t_0,\vc v_0)$. In other words, the set $\mathcal M_0$ is 
invariant and the global attractor of the system. For $0<\epsilon\ll 1$, 
the manifold $\mathcal M_0$ deforms into a nearby manifold $\mathcal M_\epsilon$
which is also invariant and globally attracting. However, $\mathcal M_\epsilon$ 
is not necessarily a collection of fixed points.
}
\label{fig:SlowMan_NonAut}
\end{figure}

The non-autonomous singular perturbation theory~\cite{IP_haller08,haragus2010} guarantees the following.
For sufficiently small $\epsilon>0$, the manifold $\mathcal M_0$ deforms into a nearby normally attracting invariant manifold
$\mathcal M_\epsilon$ which is $\mathcal O(\epsilon)$-close to 
the critical manifold $\mathcal M_0$. The trajectories of the Nesterov flow converge exponentially fast towards the 
invariant manifold  $\mathcal M_\epsilon$.
Furthermore, $\mathcal M_\epsilon$ is a graph over the slow variables $(\vc x,t)$. We write 
this graph as a formal functional series expansion in $\epsilon$, 
\begin{equation}
\mathcal M_\epsilon = \left\{(\vc x,\vc v,t)\in \mathcal X\times\mathbb R^n\times [t_0,t_0+T]: \vc v=\sum_{k=1}^\infty \epsilon^k \vc g_k(\vc x,t)\right\},
\end{equation}
where $\vc g_k:\mathcal X\times [t_0,t_0+T]\to \mathbb R^n$.
By differentiating this expression with respect to $\tau$, we obtain
\begin{align}
\vc v ' & = \sum_{k=1}^\infty \epsilon^k \left( \nabla\vc g_k(\vc x,t)\vc x'+\pard{\vc g_k}{t}\Big|_{(\vc x ,t)}t'\right)\nonumber\\
& = \sum_{k=1}^{\infty}\sum_{j=1}^\infty \epsilon^{k+j+1}\nabla \vc g_k\vc g_j + \sum_{k=1}^\infty\epsilon^{k+1} \pard{\vc g_k}{t}
\nonumber\\
& = \epsilon^2 \pard{\vc g_{1}}{t} + 
\sum_{k=3}^\infty \epsilon^k\left( \pard{\vc g_{k-1}}{t} + \sum_{j=1}^{k-2} \nabla \vc g_j\vc g_{k-j-1}\right).
\end{align}

On the other hand, using the fact that $\vc v' = -t^{-1}\vc v -\epsilon\nabla f(\vc x)$, we obtain
\begin{equation}
\vc v' = \epsilon \left(-\frac{1}{t}\vc g_1(\vc x,t) -\nabla f(\vc x)\right) - \sum_{k=2}^{\infty} \epsilon^k \frac{\vc g_k(\vc x,t)}{t}.
\end{equation}
Equating these two expressions, we obtain
\begin{subequations}
\begin{equation}
\vc g_1(\vc x,t) = - t\nabla f(\vc x), \quad \vc g_2(\vc x,t) = t\nabla f(\vc x),
\end{equation}
\begin{equation}
\vc g_k (\vc x,t) = -t\left[ \pard{\vc g_{k-1}}{t} +\sum_{j=1}^{k-2}\nabla \vc g_j\vc g_{k-j-1}\right]_{(\vc x,t)},\quad k\geq 3.
\end{equation}
\label{eq:nonAut_g}
\end{subequations}
Note that functions $\vc g_k$ include derivatives of $f$ up to order $\ceil{k/2}$.
The above results are summarized in the following theorem. 
\begin{thm}\label{thm:slowFlow_nonAut}
Let $f\in C^{r}(\mathcal X)$ with $r\geq 1$.
There exists $\rho_0>0$ such that for all $\rho >\rho_0$ the following holds.
\begin{enumerate}[(i)]
\item The trajectories of the Nesterov equation~\eqref{eq:nesterov_cont}
converge exponentially fast to an $(n+1)$-dimensional invariant manifold embedded 
in the $(2n+1)$-dimensional extended phase space 
$(\vc x,\dot{\vc  x},t)\in\mathcal X\times\mathbb R^n\times [t_0,t_0+T]$ for $0<t_0,T<\infty$.
Furthermore, this invariant manifold is a graph over the $(\vc x,t)$ coordinates 
(See figure~\ref{fig:SlowMan_NonAut} for an illustration).

\item The flow of~\eqref{eq:nesterov_cont} on the slow manifold $\mathcal M_\epsilon$ is given by
\begin{equation}
\dot{\vc  x} = \sum_{k=1}^{2r}\epsilon^{k}\vc g_{k}(\vc x,t)+\mathcal O(\epsilon^{2r+1}),
\end{equation}
where $\epsilon=\rho^{-1}$ and the maps $\vc g_k:\mathcal X\times[t_0,t_0+T]\to\mathbb R^n$ are defined recursively by~\eqref{eq:nonAut_g}.
\end{enumerate}
\end{thm}


\begin{rem}
Here, we consider the critical manifold $\mathcal M_0$ over the finite time interval $[t_0,t_0+T]$ for some finite $T>0$. 
As a result, the critical manifold is a graph over the compact domain $\mathcal X\times [t_0,t_0+T]$. 
It is tempting to extend the  slow manifold over the infinite time interval $[t_0,\infty)$. However, the fixed-point 
argument that guarantees the persistence of the slow manifold $\mathcal M_\epsilon$ is not generally valid over non-compact
sets~\cite{jones1995}. 
\label{rem:finiteT}
\end{rem}

Note that to the leading order, we have
\begin{equation}
\frac{\id \vc x}{\id t} = \vc v = -\epsilon t\nabla f(\vc x).
\label{eq:nonAut_1st_01}
\end{equation}
Reparameterizing time, by defining $\hat t = \epsilon\, t^2/2$, we obtain the usual gradient descent
\begin{equation}
\frac{\id \vc x}{\id \hat t} = -\nabla f(\vc x).
\label{eq:nonAut_1st_02}
\end{equation}
The trajectories of the two systems~\eqref{eq:nonAut_1st_01} and~\eqref{eq:nonAut_1st_02} are identical since one is a reparameterization of the other.
Therefore, to the first order, the Nesterov flow reduced to the slow manifold coincides with the usual gradient descent. 
Recall that the same conclusion held for the autonomous case (cf. Remark~\ref{rem:gradient}). 

Furthermore, examining the terms $\vc g_k(\vc x,t)$ with $k\geq 2$, 
the non-autonomous slow manifold approximation contains higher-order terms 
with $t\nabla f(\vc x)$. Collecting all such terms, we obtain
\begin{equation}
\frac{\id \vc x}{\id t}= \vc v = t\nabla f(\vc x)\sum_{k=1}^{2r} (-\epsilon)^{k} = -S_r(\epsilon) t\nabla f(\vc x),
\label{eq:nonAut_reduced_all}
\end{equation}
where $S_r(\epsilon)=\epsilon\left(1-\epsilon^{2r}\right)/(1+\epsilon)$, using the geometric series formula. 
If $f\in C^\infty(\mathcal X)$ and $\epsilon<1$, we have $S_r(\epsilon)\to \epsilon/(1+\epsilon)$ as $r\to\infty$. 
Again, reparameterizing time by defining $\hat t = S_r(\epsilon) t^2/2$, we obtain the usual gradient flow~\eqref{eq:nonAut_1st_02}.

Recall from Remark~\ref{rem:finiteT} that the existence of the slow manifold $\mathcal M_\epsilon$ is 
guaranteed over the finite time interval $t\in[t_0,t_0+T]$. This prohibits asymptotic analysis of the 
reduced flow trajectories in the limit $t\to \infty$. However, in optimization, it suffices to 
ensure that the reduced flow~\eqref{eq:nonAut_reduced_all} reaches a small neighborhood 
$B_\delta(\vc x^\ast)$ of the minimizer $\vc x^\ast$, where $0<\delta\ll 1$ is the prescribed optimization tolerance
(here, $B_\delta(\vc x^\ast)$ denotes the ball of radius $\delta$ in $\mathcal X$ centered at $\vc x^\ast$).
The set $B_\delta(\vc x^\ast)$ can be reached in finite time avoiding the need for the existence of the slow manifold
for infinite times. 

\section{Examples}\label{sec:examples}
In this section, we demonstrate our results on three functions as listed in Table~\ref{tab:examples}
and plotted in figure~\ref{fig:ex_f}. These functions are convex with their global minima at the origin, $\vc x^\ast=\vc 0$. 
We divide our numerical results into two parts: 
Section~\ref{sec:examples_aut} contains the results for the autonomous case discussed in Section~\ref{sec:slowManRed_aut};
Section~\ref{sec:examples_nonaut} corresponds to the non-autonomous 
Nesterov flow discussed in Section~\ref{sec:slowManRed_nonaut}.
\renewcommand{\arraystretch}{1.5}
\begin{table}[h!]
	\centering
	\caption{Three functions used as examples here to demonstrate the results. 
		The parameters $\mu$ and $\eta$ refer to the parameters defined in Theorem~\ref{thm:slowFlow_eucl}. 
		The parameter $\rho$ denotes the constant in the Nesterov flow~\eqref{eq:nesterov_cont},
		and $\vc x=(x_1,x_2)$.}
	\begin{tabular}{cc|c|c|c|}
		\cline{3-5}
		       & & \multicolumn{2}{ c| }{Autonomous} & Non-autonomous \\ \hline
		\multicolumn{1}{ |c| }{Example \#} & $f(\vc x)$ & \hspace{.25cm}$\mu$\hspace{.25cm} & $\eta$ & $\rho$ \\ \hline
		\multicolumn{1}{ |c| }{1} & $\frac12 \left( x_1^2+5 x_2^2\right)$ & 8 & 1 & 3\\ \hline
		\multicolumn{1}{ |c| }{2} & $\frac{1}{4} \left( x_1^4+50 x_2^4\right)$ & 2 & 1 & 1.5\\ \hline
		\multicolumn{1}{ |c| }{3} & $\log(e^{x_1^2}+e^{4x_2^2})$ & 4 & 1 & 3\\ \hline
	\end{tabular}
	\label{tab:examples}
\end{table}
\renewcommand{\arraystretch}{1}

\subsection{Autonomous case}\label{sec:examples_aut}
The numerical results presented in this section correspond to the autonomous case
discussed in Section~\ref{sec:slowManRed_aut}. Recall that the autonomous case 
assumes that the coefficients $\mu =e^{\alpha_t}-\dot{\alpha}_t$ and $\eta = e^{2\alpha_t+\beta_t}$
are time-independent. For each example, we choose a different combination of these constant 
coefficients as listed in Table~\ref{tab:examples}. While the results are valid for all $\mu>\mu_0$, 
the parameter $\mu_0$ is not a priori known (see Remark~\ref{rem:mu0}). 
In example 2, we set $\epsilon=\mu^{-1}=0.5$ to demonstrate that the slow manifolds may persist even for relatively large values of the perturbation parameter $\epsilon$.
We also note that for such larger values of $\epsilon$ (such as the one in example 2), 
the trajectories of the accelerated gradient flow exhibit a more oscillatory behavior compared to smaller $\epsilon$ where the oscillations are mostly damped. 

\begin{figure}
\centering
\subfigure[Example 1]{\includegraphics[width=.32\textwidth]{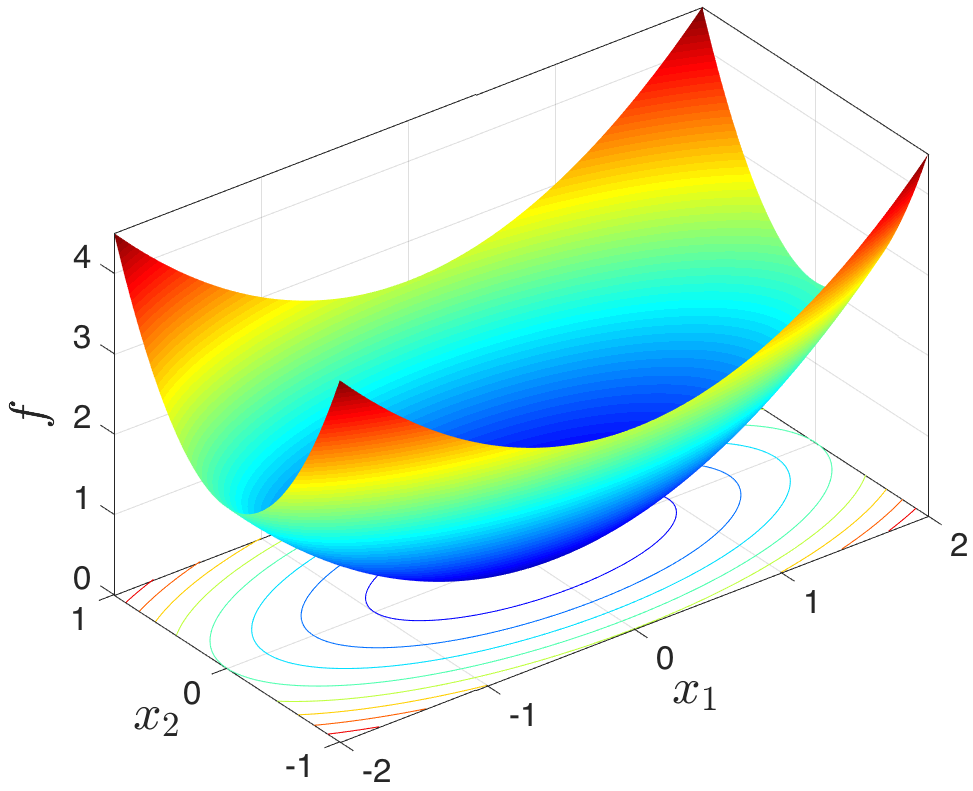}}
\subfigure[Example 2]{\includegraphics[width=.32\textwidth]{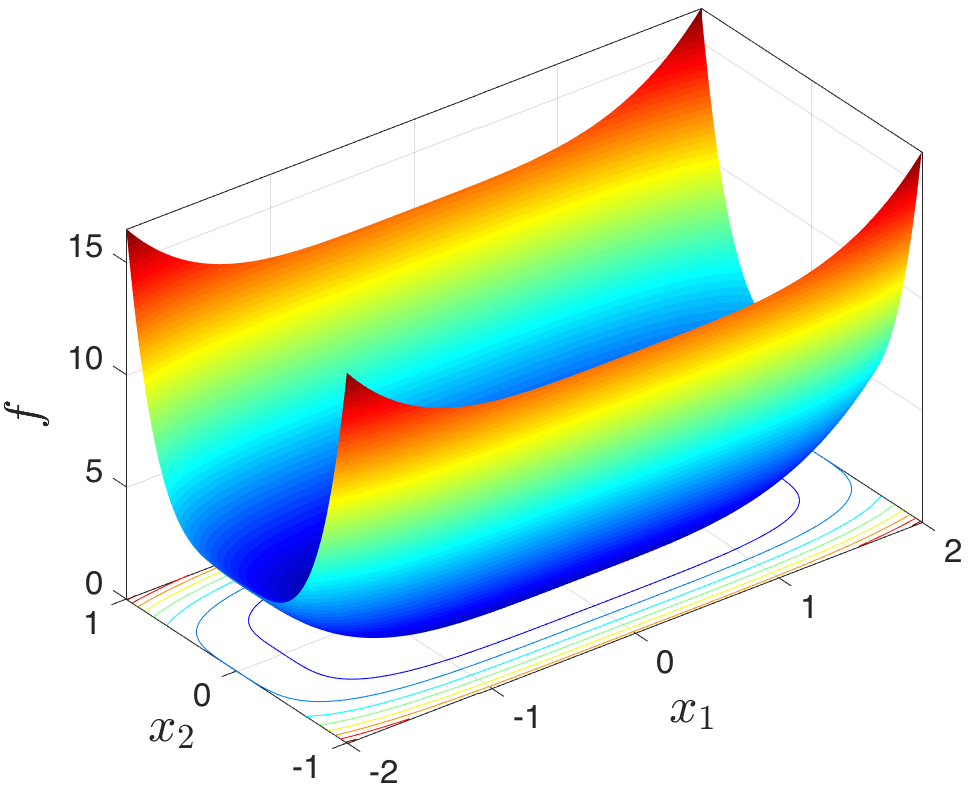}}
\subfigure[Example 3]{\includegraphics[width=.32\textwidth]{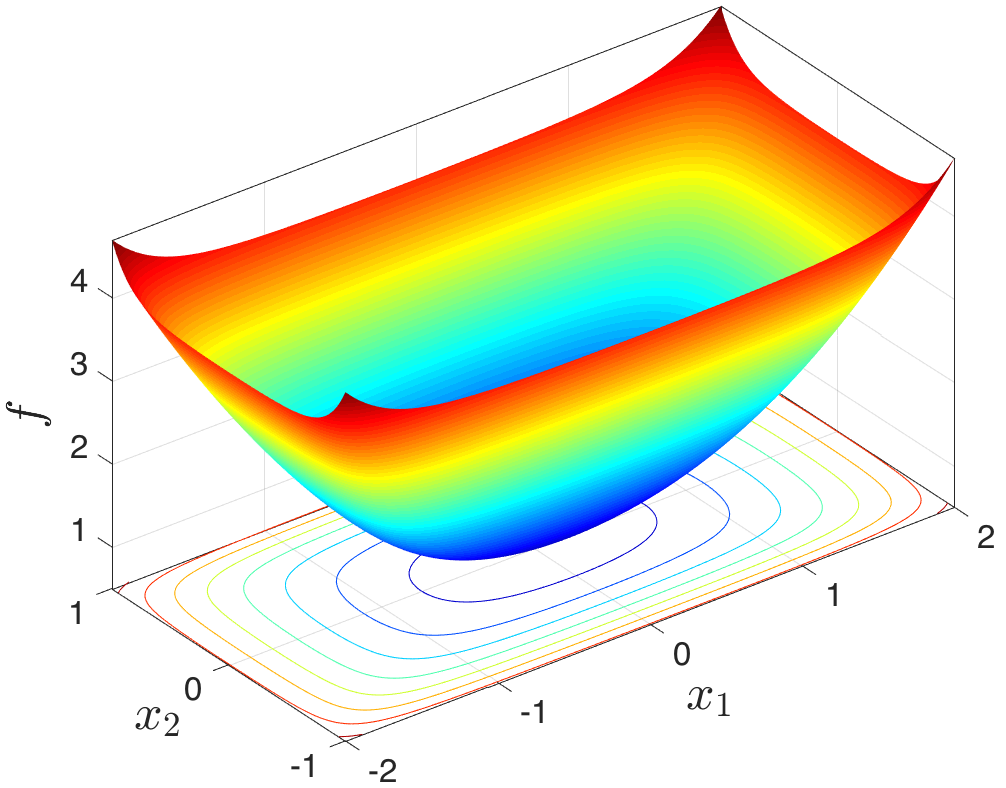}}
\caption{The functions $f$ corresponding to the examples listed in Table~\ref{tab:examples}.}
\label{fig:ex_f}
\end{figure}

Figure~\ref{fig:ex_phasespace} shows the corresponding truncated slow manifolds $\mathcal M_{\epsilon,p}$
for each example.
Here, $\mathcal M_{\epsilon,p}$ denotes the slow manifold $\mathcal M_\epsilon$ truncated to the $p$-th order. 
More specifically, the truncated slow manifold $\mathcal M_{\epsilon,p}$ is a graph over the $\vc x$-plane.
This graph, denoted by $\vc v^{\epsilon,p}:\mathcal X\to\mathbb R^n$, is defined by
\begin{equation}
\vc v^{\epsilon,p}(\vc x) \triangleq \sum_{k=1}^p \epsilon^k \vc g_k(\vc x).
\label{eq:v_trunc}
\end{equation}
where $\vc g_k$'s are given in~\eqref{eq:gk}.
For example 1, we plot the first-order truncation ($p=1$) and, for the other two examples, we plot the third-order truncations
($p=3$). Recall that the even terms in the series vanish so that the $p=3$ truncation only contains two terms.
In all three examples, the difference between the fist and third order truncations is insignificant 
and visually unnoticeable.

In each panel of figure~\ref{fig:ex_phasespace}, two types of trajectories are shown. The blue curves mark the
trajectories of the second-order Euler--Lagrange equation~\eqref{eq:EL_eucl} while the 
red curves mark the trajectories of the first-order reduced equation~\eqref{eq:reducedFlow} plotted on the slow manifold.
\begin{figure}
\centering
\subfigure[Example 1]{\includegraphics[width=.48\textwidth]{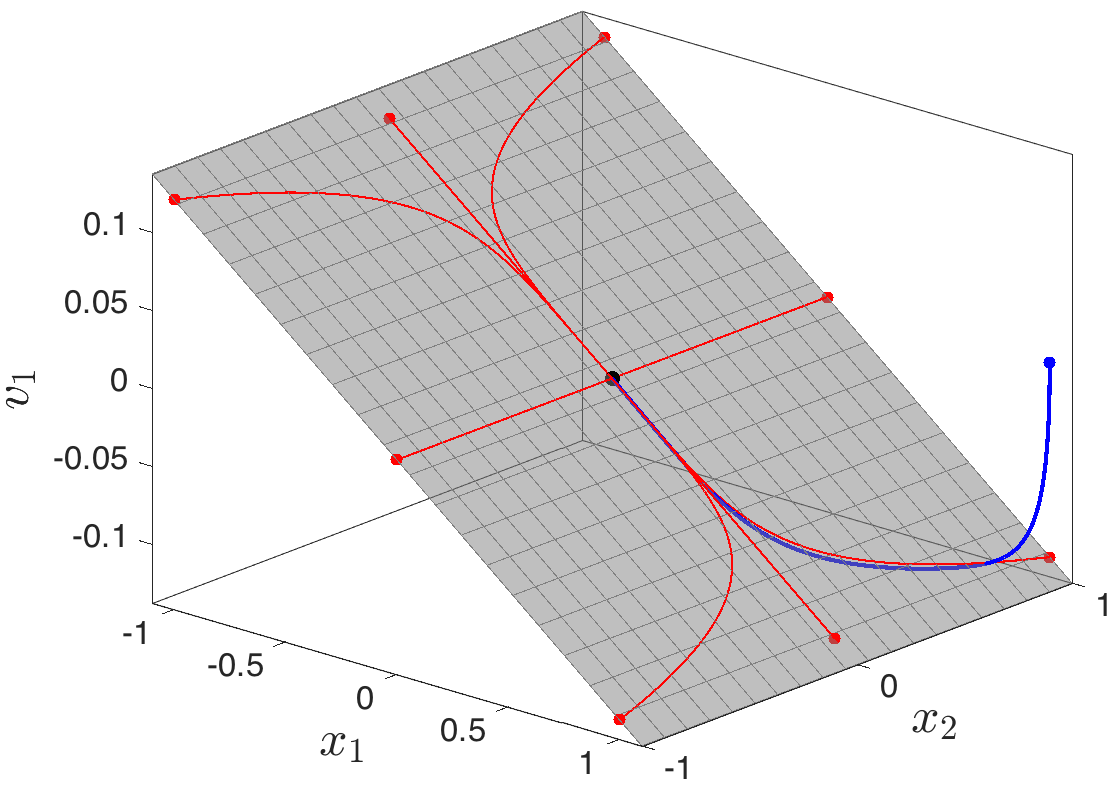}}
\subfigure[Example 2]{\includegraphics[width=.48\textwidth]{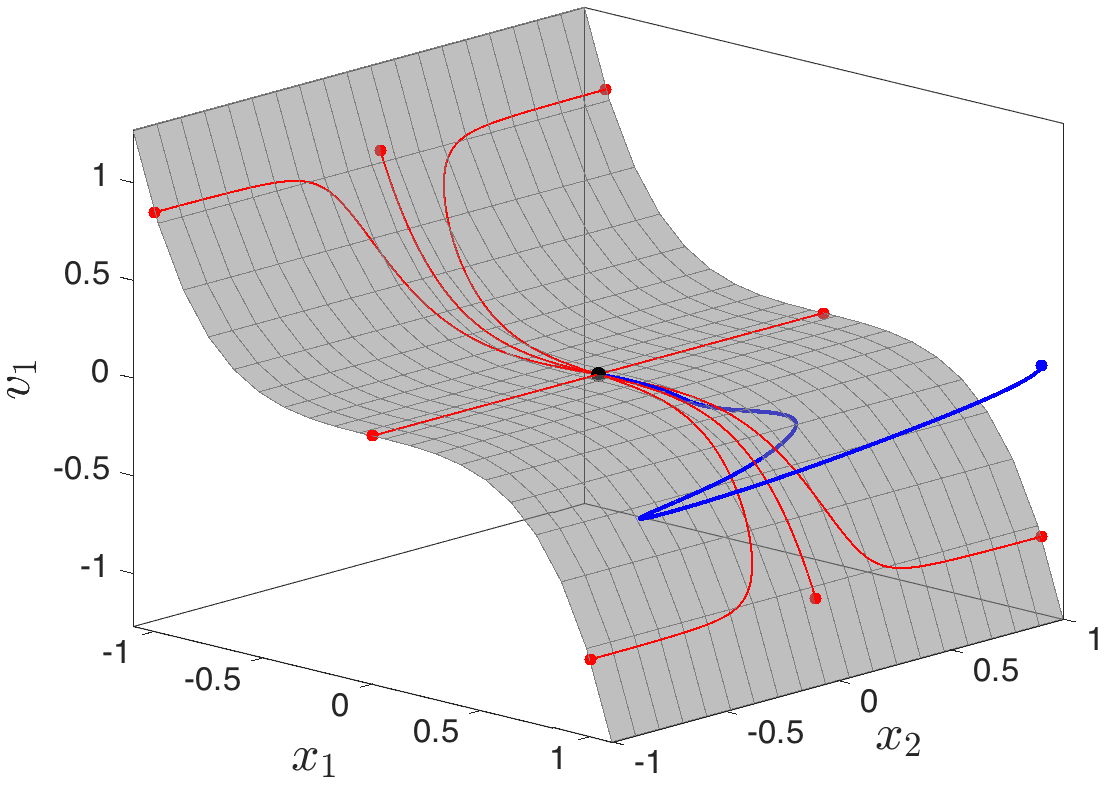}}
\subfigure[Example 3]{\includegraphics[width=.48\textwidth]{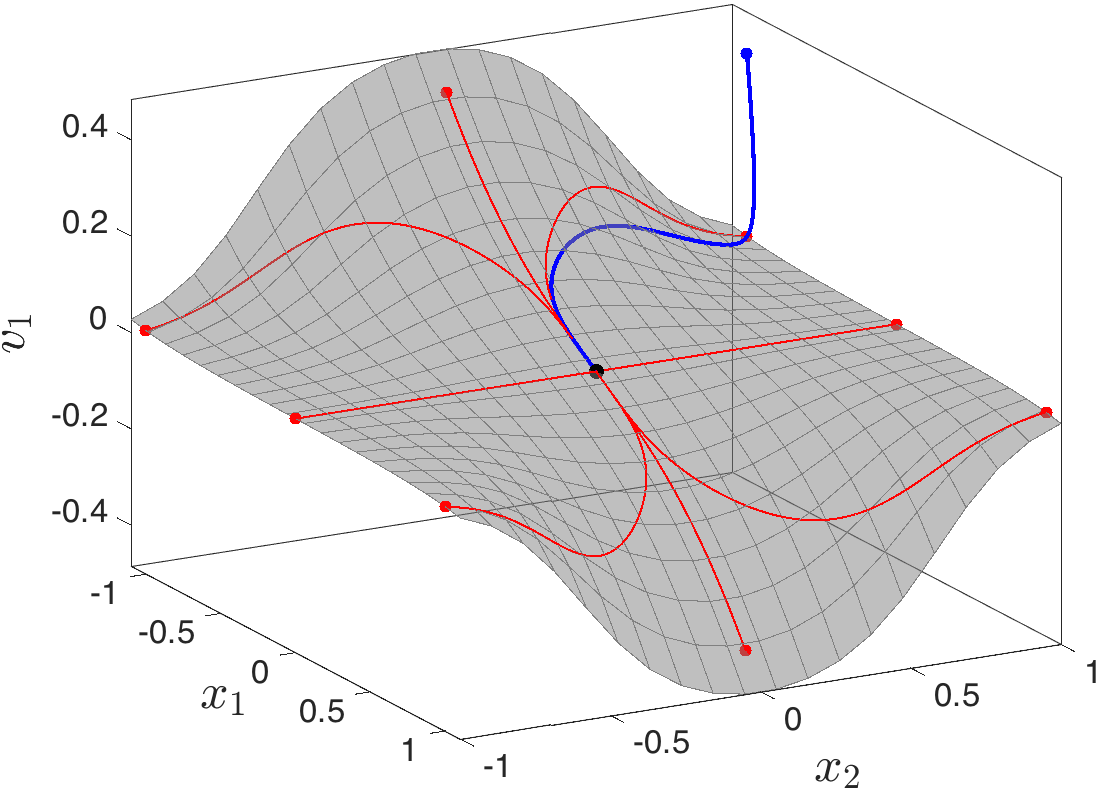}}
\caption{The projection of the slow manifolds $\mathcal M_\epsilon$ (gray surfaces) unto the $(x_1,x_2,v_1)$ subspace.
One trajectory of the Euler--Lagrange equation~\eqref{eq:EL_eucl} (blue) 
and several trajectory of the reduced system~\eqref{eq:reducedFlow} (red) are also shown. 
The initial conditions for each trajectory are marked by blue and red dots respectively. 
The black dot marks the minimum $\vc x^\ast$ (which coincides with the origin).
}
\label{fig:ex_phasespace}
\end{figure}

The trajectories of the Euler--Lagrange equation (blue curves) undergo two stages. First they approach 
the slow manifold exponentially fast. In the second stage, they closely follow the slow manifold towards the 
minimizer $\vc x^\ast$. 
These two stages are demonstrated in figure~\ref{fig:ex_K1K2} which shows the distance from the 
slow manifold $\mathcal M_{\epsilon,p}$ along the trajectory $(\vc x(t),\vc v(t))$ of the Euler--Lagrange equation
corresponding to example 1. This distance is computed as 
\begin{equation}
d(t) = \|\vc v(t) - \vc v^{\epsilon,p}(\vc x(t))\|.
\end{equation}
The quantity $d(t)$ shows two exponential slopes. Within the first time unit (see the inset of figure~\ref{fig:ex_K1K2}),
the distance drops sharply which corresponds to the convergence towards
the slow manifold. The rate of convergence seems identical for both truncations $p=1$ and $p=3$. However, the 
distance corresponding to $p=3$ decreases by a grater amount since
this higher-order truncation more accurately approximates the true slow manifold $\mathcal M_\epsilon$.
Later ($t>10$), the distance $d(t)$ continues to decrease but at a lower rate. 
We attribute this second decaying stage to the asymptotic convergence of the Euler--Lagrange solutions
to the minimizer $\vc x^\ast$ within the slow manifold.
\begin{figure}
\centering
\includegraphics[width=.6\textwidth]{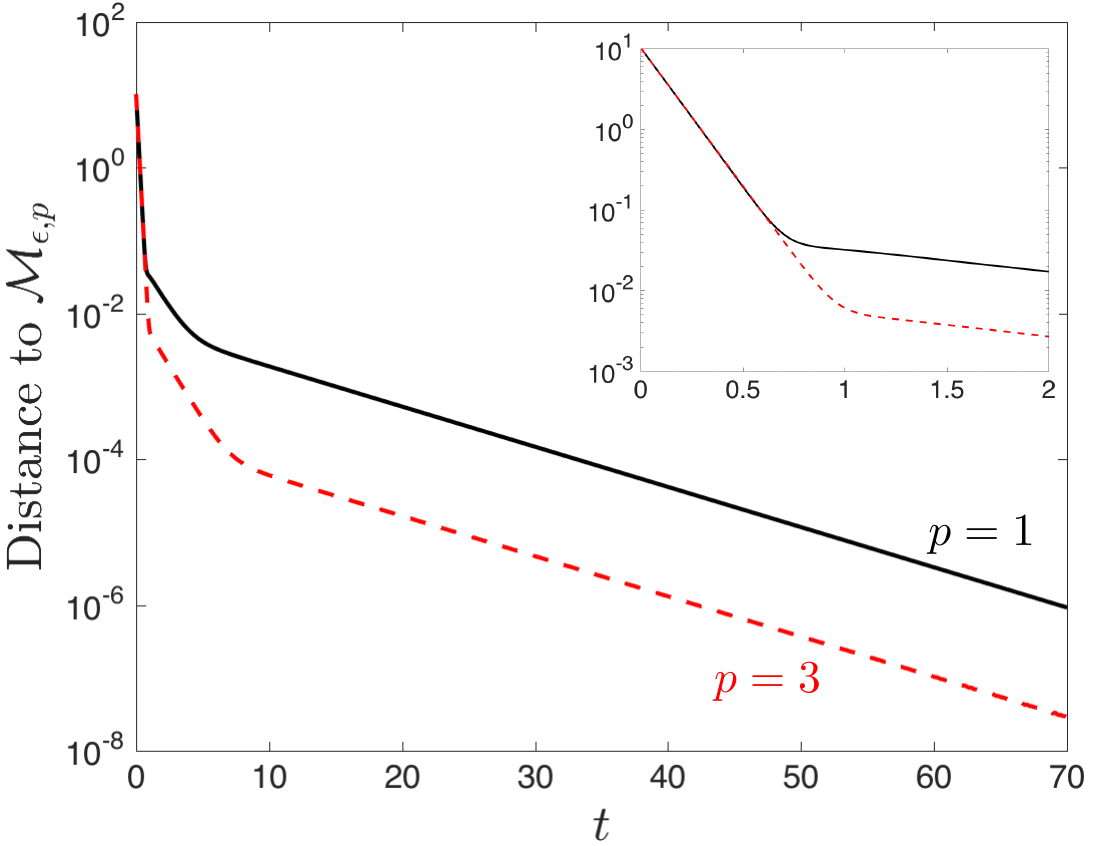}
\caption{
The distance from the truncated slow manifold $\mathcal M_{\epsilon,p}$ along the trajectories of the Euler--Lagrange equation~\eqref{eq:EL_eucl} for example 1. 
The inset shows a closeup view.}
\label{fig:ex_K1K2}
\end{figure}

A practical implication of our slow manifold reduction is that one can reduce the computational cost 
of accelerated gradient flow by skipping the first stage of their evolution (i.e. the convergence towards the slow manifold). 
Since the accelerated gradient flows are second-order differential equations, 
they require an initial guess $(\vc x_0,\vc v_0)$ as their initial condition.
Using the slow manifold, it is advantageous to choose the alternative initial velocity
$\tilde{\vc v}_0 = \vc v^{\epsilon,p}(\tilde{\vc x}_0)$ where
\begin{equation}
\tilde{\vc x}_0 = \arg\min_{\vc x\in\mathcal X}\left[ |\vc x-\vc x_0|^2+|\vc v^{\epsilon,p}(\vc x)-\vc v_0|^2\right].
\label{eq:min_ic}
\end{equation}
The point $(\tilde{\vc x}_0,\tilde{\vc v}_0)$ is the closest point on the slow manifold $\mathcal M_{\epsilon,p}$ 
to the initial guess $(\vc x_0,\vc v_0)$.
This alternative initial guess $(\tilde{\vc x}_0,\tilde{\vc v}_0)$ avoids the initial decay phase of the 
flow (i.e. convergence towards the slow manifold) by placing the trajectory $\mathcal O(\epsilon^p)$-close to this manifold at the 
initial time. As a result, several evaluations of the function $f$ and its derivatives during the decay phase are dispensed with. 

\begin{figure}
\centering
\includegraphics[width=.6\textwidth]{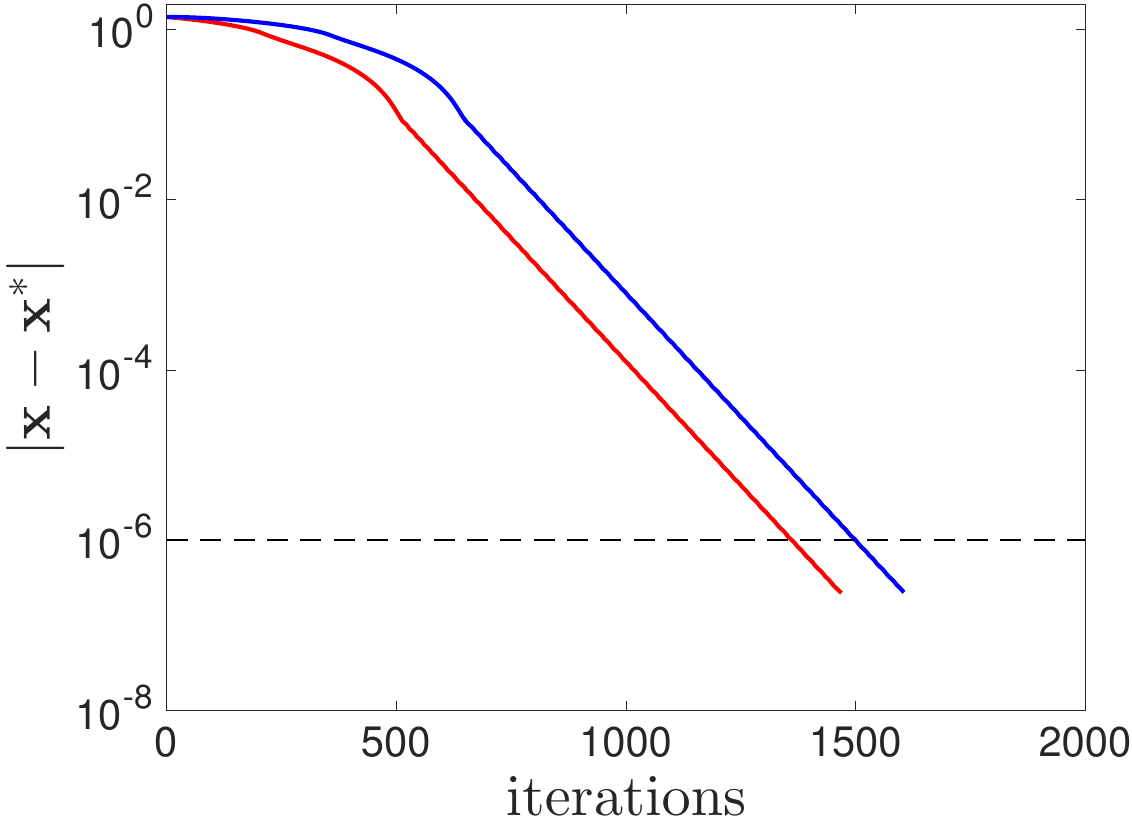}
\caption{Alternative initial conditions. The function in example 1 is minimized using the 
accelerated gradient descent~\eqref{eq:EL_eucl}. The blue curve corresponds to the initial conditions
$\vc x_0 = (1,1)$ and $\vc v_0 = (0,0)$. The red curve corresponds to the alternative initial conditions
$\tilde{\vc x}_0 = \vc x_0 = (1,1)$ and $\tilde{\vc v}_0 = \vc v^{\epsilon,1}(\vc x_0)\simeq (-0.125,-0.625)$.
It takes fewer iterations for the red curve to reach the error tolerance $10^{-6}$ marked by the dashed black curve.
}
\label{fig:cost_saved}
\end{figure}

However, carrying out the minimization~\eqref{eq:min_ic} can be costly itself, outweighing 
the saved computational cost from skipping the decay phase.
Here, we propose a cheaper approach by choosing the alternative initial guess
$\tilde{\vc x}_0=\vc x_0$ and $\tilde{\vc v}_0=\vc v^{\epsilon,p}(\vc x_0)$
for a given $\vc x_0\in\mathcal X$. 
For $p=1$, for instance, this only requires 
one computation of the gradient of $f$ since $\vc v^{\epsilon,1}(\vc x_0) = -\epsilon\eta \nabla f(\vc x_0)$.

Figure~\ref{fig:cost_saved} show an implementation of this alternative initial condition on example 1. 
We solve the accelerated gradient flow~\eqref{eq:EL_eucl} from two different initial conditions. 
First, we set $\vc x_0 = (1,1)$ and $\vc v_0 = (0,0)$. This corresponds to the blue curve 
in figure~\ref{fig:cost_saved}. Then we choose the alternative initial conditions $\tilde{\vc x}_0=\vc x_0=(1,1)$
and $\tilde{\vc v}_0 =\vc v^{\epsilon,1}(\vc x_0)=  -\epsilon\eta \nabla f(\vc x_0)\simeq (-0.125,-0.625)$. This is marked by the 
red curve. The alternative initial guess $(\tilde{\vc x}_0,\tilde{\vc v}_0)$ takes fewer iterations to reach a given error $|\vc x-\vc x^\ast|$.

For instance, it takes $1499$ iterations to reach the error $|\vc x-\vc x^\ast|<10^{-6}$ if we start from the initial 
guess $(\vc x_0,\vc v_0)$. To reach the same error tolerance, it takes $1361$ iterations if we start from the 
alternative initial guess $(\tilde{\vc x}_0,\tilde{\vc v}_0)$. Accounting for the gradient evaluation required to 
compute $\tilde{\vc v}_0$, the alternative initial guess takes $137$ less gradient evaluations
to reach the error tolerance $10^{-6}$. In high dimensions this 
can amount to a noticeable reduction in the computational time. 

\subsection{Non-autonomous case}\label{sec:examples_nonaut}
The numerical results presented in this section correspond to the non-autonomous case
discussed in Section~\ref{sec:slowManRed_nonaut}. The parameter $\rho =\epsilon^{-1}$ is 
reported in the last column of Table~\ref{tab:examples}. For examples 1 and 3, we use the standard value $\rho=3$. To demonstrate that the slow manifold reduction may be valid for even smaller values
of $\rho$ (or equivalently larger values of $\epsilon$), we set $\rho=1.5$ in example 2.

The numerical results are reported in figure~\ref{fig:nonaut}. In all examples, the initial conditions
are $\vc x_0=(1,1)$ and $t_0=0.1$. The black curves mark the trajectories of the full 
Nesterov flow~\eqref{eq:nesterov_cont} with the initial velocity $\vc v_0=\dot{\vc x}(t_0)=(2,0)$.
Note that the phase space $\mathcal X\times\mathbb R^2\times [t_0,t_0+T]$ of the Nesterov flow  
is five-dimensional. Therefore, the bottom panel in figure~\ref{fig:nonaut} shows the projections
of the trajectories on the three-dimensional subspace $\mathcal X\times [t_0,t_0+T]$.
\begin{figure}
\centering
\includegraphics[width=.98\textwidth]{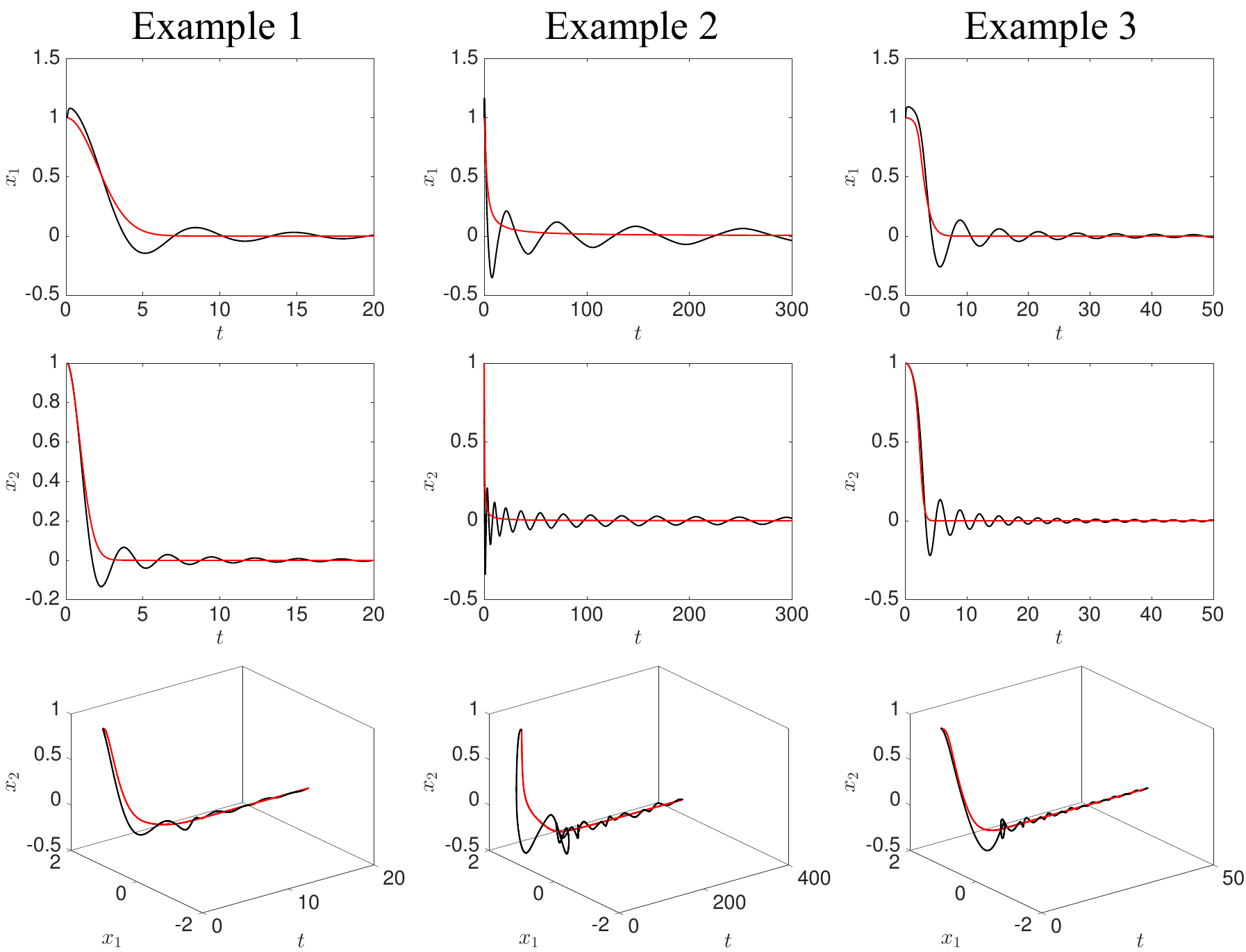}
\caption{Trajectories of the accelerated gradient flows. Example numbers 
refer to the functions listed in Table~\ref{tab:examples}. 
The black curves correspond to the trajectories of the Nesterov flow~\eqref{eq:nesterov_cont}
and the red curves correspond to the reduced-order flow~\eqref{eq:nonAut_reduced_all}.
}
\label{fig:nonaut}
\end{figure}

The red curves mark the trajectories of the slow manifold reduced flow~\eqref{eq:nonAut_reduced_all}.
The trajectories of both systems (the reduced flow and the full Nesterov flow) converge 
to the global minimum $\vc x^\ast =\vc 0$. As expected, the trajectories of the Nesterov flow oscillate around the trajectories of the reduced flow. As time increases, the amplitude of these oscillations
decay and the Nesterov flow trajectories converge onto the slow manifold trajectories.

We have repeated the simulations for various initial conditions $(\vc x_0,\vc v_0)$
and observed very similar behavior (namely, the convergence of the
oscillatory Nesterov trajectories onto the slow manifold trajectories), 
ensuring that the results are not sensitive to the initial conditions.

\section{Conclusions}\label{sec:conc}
It has recently been shown that the continuous-time limit of accelerated gradient descent methods
are the Euler--Lagrange equations corresponding to a single Lagrangian. This Euler--Lagrange equation
takes the form of a second-order ordinary differential equation that we refer to as the accelerated gradient flow.

Here we showed that, under certain assumptions, the accelerated gradient flows
possess an attracting, invariant, slow manifold. The trajectories of accelerated gradient flow
undergo two stages. First, they converge exponentially fast towards the slow manifold. 
Then they closely follow the flow within the slow manifold. We derived a general explicit
formula that approximates the slow manifold to any arbitrary order of accuracy. 

To the leading order, the flow within the slow manifold coincides with the
usual gradient (or steepest) descent. Higher order approximations of the slow manifold, however, involve 
higher order derivatives of the objective function.

We divided our analysis into two parts. In the first part (section~\ref{sec:slowManRed_aut}), 
it is assumed that the accelerated gradient flow is autonomous (i.e., it does not have any explicit dependence on time).
In this setting, the classical geometric singular perturbation theory was applied to prove the
existence of the slow manifold as an $n$-dimensional submanifold of the phase space. 
The second part (section~\ref{sec:slowManRed_nonaut}) concerns the Nesterov accelerated
gradient flow which is a non-autonomous differential equation. 
More recent singular perturbation results are applicable to this flow. 
While the conclusions are similar to the autonomous case, the Nesterov 
flow has an $(n+1)$-dimensional slow manifold in the extended phase space.

A practical implication of our results is the reduced computational cost of 
accelerated gradient flows by initializing it close to the slow manifold
which is now known explicitly. This initialization avoids the initial stage of the 
flow, which involves converging towards the slow manifold, hence reducing the computational cost.

Since accelerated gradient iterations are temporal discretizations of the 
accelerated gradient flow~\cite{wibisono2016}, we expect similar slow manifold reductions to hold for
the discrete accelerated methods. However, a rigorous singular perturbation analysis of these 
iterations is desirable and will be pursued in future work.


\section*{Acknowledgments}
I am grateful to Andre Wibisono (Georgia Tech) and Tony Roberts (University of Adelaide) for their valuable comments 
on an earlier draft of this paper.
I also acknowledge fruitful discussions with George Haller (ETH Zurich) on singular perturbation theory
for non-autonomous systems. I am grateful to two anonymous referees whose constructive comments 
significantly improved the paper.

\begin{appendices}
\section{Proof of Theorem~\ref{thm:assymptStable}}\label{app:proof_asymStable}
To prove this theorem, we will need the following lemma. 
\begin{lem}\label{lem:sym}
The matrices $\nabla \vc g_k(\vc x)\in\mathbb R^{n\times n}$ are symmetric for any $\vc x\in\mathcal X$.
\end{lem}
\begin{proof}
To show that $\nabla \vc g_k(\vc x)$ are symmetric, we first show that $\vc g_k=\nabla \phi_k$
for twice continuously differentiable functions $\phi_k:\mathcal X\to\mathbb R$.
And therefore $\nabla \vc g_k = \nabla^2\phi_k$ is symmetric. 
First, note that $\vc g_{2k}=\vc 0$ and therefore this assertion is trivially correct for 
even indices with $\phi_{2k}=0$. For odd indices, we proceed with induction.

Note that $\vc g_1 = - \eta \nabla f$ and therefore we have $\phi_1(\vc x) = -\eta f(\vc x)$.
Now assume that $\vc g_\ell = \nabla \phi_\ell$ for all $1\leq\ell<2k+1$. 
This implies
\begin{align}
\vc g_{2k+1}&=-\sum_{\ell=1}^{2k-1} \nabla^2\phi_\ell \nabla\phi_{2k-\ell}\nonumber\\
&=\nabla\left[-\frac12 \sum_{\ell=1}^{2k-1}\langle \nabla \phi_\ell,\nabla \phi_{2k-\ell}\rangle \right]\nonumber\\
&=\nabla\left[-\frac12 \sum_{\ell=1}^{2k-1}\langle \vc g_\ell,\vc g_{2k-\ell}\rangle \right].
\end{align}
The second line, in the above equation, follows from the series of identities,
\begin{align}
\nabla\left[-\frac12 \sum_{\ell=1}^{2k-1}\langle \nabla \phi_\ell,\nabla \phi_{2k-\ell}\rangle \right]&=
-\frac12 \sum_{\ell=1}^{2k-1}\left[ \nabla^2\phi_\ell \nabla \phi_{2k-\ell} +  \nabla^2\phi_{2k-\ell} \nabla \phi_{\ell}\right]\nonumber\\
&=-\frac12 \sum_{\ell=1}^{2k-1} \nabla^2\phi_\ell \nabla \phi_{2k-\ell}
-\frac12 \sum_{\ell=1}^{2k-1} \nabla^2\phi_{2k-\ell} \nabla \phi_{\ell}\nonumber\\
&=-\frac12 \sum_{\ell=1}^{2k-1} \nabla^2\phi_\ell \nabla \phi_{2k-\ell}
-\frac12 \sum_{j=1}^{2k-1} \nabla^2\phi_{j} \nabla \phi_{2k-j}\nonumber\\
&=- \sum_{\ell=1}^{2k-1} \nabla^2\phi_\ell \nabla \phi_{2k-\ell},
\end{align}
where on the third line we used the change of indices $2k-\ell \mapsto j$.
Therefore, $\vc g_k=\nabla \phi_k$ for all $k\geq 1$, where
the scalar functions $\phi_k:\mathcal X\to\mathbb R$ are defined recursively by
$\phi_1 = -\eta f$, $\phi_{2k}=0$ and
\begin{equation}
\phi_{2k+1}=-\frac12 \sum_{\ell=1}^{2k-1}\langle \nabla \phi_\ell,\nabla \phi_{2k-\ell}\rangle,\quad 
k\geq 1.
\end{equation}
\end{proof}

\begin{proof}[Proof of Theorem~\ref{thm:assymptStable}]
Consider the Lyapunov function
\begin{equation}
\mathcal E(\vc x) = \epsilon\eta \left(f(\vc x)-f(\vc x^\ast)\right) +\sum_{k=1}^{r} \frac{\epsilon^{2k+1}}{2}\sum_{\ell=1}^{2k-1} 
\langle \vc g_\ell(\vc x), \vc g_{2k-\ell}(\vc x)\rangle.
\end{equation}
Note that $\mathcal E(\vc x^\ast)=0$ and that the first term $\epsilon\eta (f(\vc x)-f(\vc x^\ast))$
is positive for all $\vc x\in\mathcal X\backslash\{ \vc x^\ast\}$. The second term (corresponding to $k=1$) is 
$\epsilon^3 \langle \vc g_1(\vc x),\vc g_1(\vc x)\rangle = \epsilon^3\eta^2\|\nabla f(\vc x)\|^2$ which is also
positive for all  $\vc x\in\mathcal X\backslash\{ \vc x^\ast\}$.
It is straightforward to show that the higher-order terms in the series are also positive at least in a neighborhood of
$\vc x^\ast$.

The gradient of the Lyapunov function is given by
\begin{align}
\nabla\mathcal E(\vc x) & = \epsilon\eta\nabla f(\vc x) +
\sum_{k=1}^{r} \frac{\epsilon^{2k+1}}{2} \sum_{\ell=1}^{2k-1}\big[\nabla \vc g_\ell(\vc x)^\top \vc g_{2k-\ell}(\vc x)
+\nabla \vc g_{2k-\ell}(\vc x)^\top \vc g_\ell(\vc x)\big] \nonumber\\
&=\epsilon\eta\nabla f(\vc x) +\sum_{k=1}^{r} \epsilon^{2k+1}\sum_{\ell=1}^{2k-1}\nabla \vc g_\ell(\vc x) \vc g_{2k-\ell}(\vc x)\nonumber\\
&=-\epsilon \vc g_1(\vc x) -\sum_{k=1}^{r} \epsilon^{2k+1}\vc g_{2k+1}(\vc x) = -\sum_{k=0}^r \epsilon^{2k+1}\vc g_{2k+1}(\vc x)
=-\dot{\vc x}.
\end{align}
In these series of identities, we used the fact that $\nabla \vc g_\ell(\vc x)$ are symmetric matrices 
(see Lemma~\ref{lem:sym}) and that
\begin{equation}
\sum_{\ell=1}^{2k-1}\nabla \vc g_{2k-\ell}(\vc x) \vc g_\ell(\vc x) = \sum_{j=1}^{2k-1} \nabla \vc g_{j}(\vc x) \vc g_{2k-j}(\vc x)
\end{equation}

Therefore we have $\frac{\id}{\id t}\mathcal E(\vc x)=\langle \nabla \mathcal E(\vc x),\dot{\vc x}\rangle=-\|\dot{\vc x}\|^2\leq 0$
with the identity attained only at the minimum $\vc x^\ast$. This completes the proof.
\end{proof}
\end{appendices}


\end{document}